\documentclass[12pt,reqno]{amsart}
\usepackage{mathtools}
\usepackage{amssymb, mathabx, enumitem,amsmath,mathrsfs}
\usepackage{color}
\usepackage[colorlinks=true,linkcolor=blue,urlcolor=blue,citecolor=blue, pagebackref]{hyperref}
\usepackage[alphabetic]{amsrefs}
\usepackage{tikz-cd}
\usepackage{tikz}
\usepackage{tikz-3dplot}
\usetikzlibrary{arrows,cd,snakes}
\usepackage{graphics}
\usepackage{arydshln}

\newcommand{\sbt}{\,\begin{picture}(-1,1)(-1,-3)\circle*{3}\end{picture}\ }

\usepackage[margin=0.7in]{geometry}

\newcommand{\C}{\mathbb{C}}

\newcommand{\Z}{\mathbb{Z}}

\newcommand{\op}{\operatorname}

\newcommand{\mf[1]}{\mathfrak{#1}}
\newcommand{\mc[1]}{\mathcal{#1}}
\newcommand{\onto}{\twoheadrightarrow}
\newcommand{\dom}{\mathcal{P}^+}

\newtheorem{theorem}{Theorem}[section]

\newtheorem{example}[theorem]{Example}
\newtheorem{remark}[theorem]{Remark}
\newtheorem{conjecture}[theorem]{Conjecture}

\newtheorem{corollary}[theorem]{Corollary}
\newtheorem{question}[theorem]{Question}
\newtheorem{proposition}[theorem]{Proposition}

\newtheorem{lemma}[theorem]{Lemma}

\newtheorem{definition}[theorem]{Definition}
\newtheorem{definition/lemma}[theorem]{Definition/Lemma}

\DeclareMathOperator{\End}{End}

\title[Hochschild cohomology of flag varieties]{Irreducible components in Hochschild cohomology of flag varieties}

\author{Sam Jeralds}
\address{School of Mathematics and Statistics, University of Sydney, Camperdown, NSW 2050, AUS}
\email{samuel.jeralds@sydney.edu.au}

\begin{document}

\begin{abstract} 
Let $G$ be a simple, simply-connected complex algebraic group with Lie algebra $\mathfrak{g}$, and $G/B$ the associated complete flag variety. The Hochschild cohomology $HH^\bullet(G/B)$ is a geometric invariant of the flag variety related to its generalized deformation theory and has the structure of a $\mathfrak{g}$-module. We study this invariant via representation-theoretic methods; in particular, we give a complete list of irreducible subrepresentations in $HH^\bullet(G/B)$ when $G=SL_n(\mathbb{C})$ or is of exceptional type (and conjecturally for all types) along with nontrivial lower bounds on their multiplicities. These results follow from a conjecture due to Kostant on the structure of the tensor product representation $V(\rho) \otimes V(\rho)$.

\end{abstract}

\maketitle

\section{Introduction} \label{Intro}

Let $G$ be a finite-dimensional, simple, simply-connected complex algebraic group with Lie algebra $\mf[g]$, and let $G/B$ be the associated complete flag variety. Developing since roughly the mid-twentieth century, the primary aim of Geometric Representation Theory has been to relate the representation theory of $G$ to the geometry of spaces like the flag variety and generalizations or subvarieties thereof. This program has since found vast impact on major themes of research--with Kazhdan-Lusztig theory, the geometric Langlands program, and categorical representation theory being a few key examples.

While this field has developed into an impressive mathematical landscape, historically Geometric Representation Theory has sought to understand the interplay between the structure or construction of representations of $G$ and \emph{geometric or topological invariants} of $G/B$. Examples of this include the Littlewood--Richardson rule, which through combinatorial methods simultaneously computes structure constants for the singular cohomology of Grassmannians and the representation ring for $GL_n(\C)$; the Borel--Weil theorem, which uniformly constructs irreducible representations of $G$ via the cohomology of line bundles on $G/B$; or Beilinson--Bernstein's groundbreaking work on the geometry of D-modules.

We continue in the spirit of the classical approach by considering another geometric invariant: the Hochschild cohomology $HH^\bullet(G/B)$ of complete flag varieties. Originating from the study of associative algebras, Hochschild cohomology captures a ``generalized deformation" theory of varieties; we recall the basic definitions and properties of this invariant in Section \ref{Hoch}. Specifically, we highlight that $HH^\bullet(G/B)$ has a natural $\mf[g]$-module structure. Unlike many other generalized cohomology theories, the Hochschild cohomology of $G/B$ is not well understood; however, this has recently been a topic of increasing interest. In \cite{BS}, the Hochschild cohomology of certain partial flag varieties (referred to as \textit{generalized Grassmannians}) was explicitly computed using Lie algebra homology techniques. In \cite{HV}, a computer algorithm was described and implemented for computing the Hochschild cohomology of flag varieties using the BGG resolution, inspired by its connection to the center of small quantum groups as given in \cite{LQ1}, \cite{LQ2}. Beyond these results or few explicit examples, little is known about the structure of $HH^\bullet(G/B)$ in general.

\subsection{$HH^\bullet(G/B)$ and Kostant's $V(\rho) \otimes V(\rho)$ conjecture}
In this note, we make a first step in understanding $HH^\bullet(G/B)$ by exhibiting a novel connection to a conjecture of Kostant. Particularly, let $V(\rho)$ be the finite-dimensional, irreducible representation of $\mf[g]$ of highest weight $\rho$, the half-sum of positive roots (see Section \ref{bundles} for notational conventions). As part of his study on the Clifford algebras of semisimple Lie algebras $\mf[g]$ \cite{Kos}, Kostant found as $\mf[g]$-representations that 
$$
\bigwedge \nolimits^{\! \bullet} \mf[g] \cong 2^\ell \left( V(\rho) \otimes V(\rho) \right),
$$
where $\ell:=\op{rank}(\mf[g])$. Attempting to understand how this representation decomposes, Kostant made the following conjecture, first recorded in the work of Berenstein--Zelevinsky \cite{BZ}. 

\begin{conjecture} \label{Kconj} Let $\lambda \in \dom$ be a dominant integral weight. Then $V(\lambda) \subset V(\rho) \otimes V(\rho)$ is an irreducible subrepresentation if and only if $\lambda \leq 2\rho$ in the usual Bruhat--Chevalley order, or dominance order, on the set of weights. 
\end{conjecture}

\noindent Of course, the ``only if" direction of this conjecture is a necessary condition; that is, the conjecture can be restated as the sufficiency of the condition $\lambda \leq 2\rho$ for $V(\lambda) \subset V(\rho) \otimes V(\rho)$. 

When $\mf[g]=\mf[sl]_n(\C)$, Conjecture \ref{Kconj} was proven in the same paper of Berenstein--Zelevinsky \cite{BZ}. Much later, Chirivi--Kumar--Maffei \cite{CKM} proved a weaker version of this conjecture for all semisimple $\mf[g]$ involving a saturation factor, which recovers the known result for $\mf[sl]_n(\C)$. They also report in loc. cit. that Conjecture \ref{Kconj} had been confirmed to hold in the exceptional types via direct computational verification. Generalizations of this conjecture and corresponding results have recently been made to other branching problems \cite{NP} and to tensor products of affine Kac--Moody Lie algebra representations \cite{JK}.

While seemingly unrelated to Hochschild cohomology, our goal for the present work is to connect Kostant's conjecture to $HH^\bullet(G/B)$. We do this using familiar tools of representation theory, such as the Borel--Weil--Bott theorem, character formulas, and Euler--Poincar\'e characteristics of vector bundles on $G/B$. As a preliminary result, we have the following proposition (cf. Proposition \ref{embedprop2}), which holds independently of Conjecture \ref{Kconj}. 

\begin{proposition} \label{embeddingprop}
For any simple, simply-connected complex algebraic group $G$, there is an embedding of representations 
$$
V(\rho)\otimes V(\rho) \hookrightarrow HH^\bullet(G/B).
$$
\end{proposition}

For a dominant integral weight $\lambda \leq 2\rho$, let $m_{\lambda}$ be the multiplicity with which $V(\lambda)$ appears in $V(\rho) \otimes V(\rho)$. Essentially as a corollary to Proposition \ref{embeddingprop}, we have the following theorem, which is the main result of this note (cf. Theorem \ref{maintheorem2}). 

\begin{theorem} \label{maintheorem}
Assume that Conjecture \ref{Kconj} holds. Then for any dominant integral weight $\lambda \in \dom$, we have that $V(\lambda) \subset HH^\bullet(G/B)$ is an irreducible component if and only if $\lambda \leq 2\rho$, and furthermore appears with nonzero multiplicity at least $m_\lambda$. 

In particular, this result holds for the cases $G=SL_n(\C)$ or $G$ of exceptional type. 
\end{theorem} 

While the assumption that $G$ is simply-connected is typically necessary for considering the representation $V(\rho)$ and Kostant's conjecture, this assumption can be dropped for Theorem \ref{maintheorem}; see Remark \ref{SimpConn}.

\subsection{Outline of paper} 
We now give an outline of the remainder of this paper. In Section \ref{bundles}, after defining our notations and conventions, we review the construction of equivariant vector bundles on $G/B$, their cohomology, and their connections to representations via their Euler--Poincar\'e characteristic and the Borel--Weil--Bott theorem. In Section \ref{Hoch}, we recall the basic definitions and properties of Hochschild cohomology of smooth projective varieties, quickly restricting to the case of $G/B$,  and its relationship to the sheaf of polyvector fields. Finally, in Section \ref{embedding} we give proofs of Proposition \ref{embeddingprop} and Theorem \ref{maintheorem} along with discussion for further topics of interest.

\subsection*{Acknowledgements} We thank Marc Besson and Josh Kiers for many helpful conversations and comments throughout the inception of this project, and Pieter Belmans for insightful correspondence on Hochschild cohomology and corrections to an earlier draft.

\section{Vector bundles on $G/B$ and the Borel--Weil--Bott theorem} \label{bundles}

We recall in this section the construction of $G$-equivariant vector bundles on flag varieties, with particular focus on the case of line bundles and the Borel--Weil--Bott theorem, and their Euler-Poincar\'e characteristics. All results of this section are well-established in the literature; we include this exposition both as a convenience for the reader as well as to fix our choice of conventions from the slew of differing-but-equivalent ones which appear elsewhere. We follow most closely the conventions of \cite{Ku2}, restricted to the finite-dimensional setting. 

To this end, let $G$ be a finite-dimensional, simple, simply-connected complex algebraic group. We also fix a maximal torus and Borel subgroup $T \subset B \subset G$. The corresponding Lie algebras are denoted $\mf[h] \subset \mf[b] \subset \mf[g]$, respectively. Let $N \subset B$ be the maximal unipotent subgroup, and denote by $\mf[n]:=[\mf[b], \mf[b]]$ its Lie algebra. The choice of maximal torus and Borel fixes a set of positive roots $\Phi^+$ in the root system $\Phi$ for $G$; we emphasize that $B$ (equivalently, $\mf[b]$) will be associated to this set of \emph{positive} roots. We also have the maximal unipotent group of $G$ opposite to $N$, which we denote by $U$, and its Lie algebra $\mf[u]$; these will thus be associated to the negative roots $\Phi^-=-\Phi^+$. 

Let $\{\alpha_i\}_{i=1}^{\ell}$ be the set of simple roots, where we set $\ell:= \op{rank}(G)$. These span the root lattice $\mc[Q]:= \bigoplus_{i=1}^\ell \Z \alpha_i$. We also denote $\mc[Q]^+:= \bigoplus_{i=1}^\ell \Z_{\geq 0} \alpha_i$, the positive root lattice. We similarly have the simple coroots $\{ \alpha_i^\vee\}_{i=1}^\ell$. We denote by $\{\varpi_i\}_{i=1}^\ell$ the fundamental weights, defined by $\varpi_i(\alpha_j^\vee)=\delta_{i,j}$. These span the integral weight lattice $\mc[P]:= \bigoplus_{i=1}^\ell \Z \varpi_i$; we denote by $\mc[P]^+:= \bigoplus_{i=1}^\ell \Z_{\geq 0} \varpi_i$ the set of dominant integral weights. We highlight in particular the dominant integral weight $\rho:=\frac{1}{2}\sum_{\beta \in \Phi^+} \beta = \sum_{i=1}^\ell \varpi_i$, the Weyl vector. 

To any dominant integral weight $\lambda \in \dom$, we can associate the finite-dimensional, irreducible $\mf[g]$-representation $V(\lambda)$ of highest weight $\lambda$; since $G$ is simply-connected, we can equivalently view this as a $G$-representation (and often make this association). These representations have weight space decompositions with respect to the $\mf[h]$-action (equivalently, $T$-action) given by 
$$
V(\lambda)=\bigoplus_{\mu \in \mc[P]} V(\lambda)_\mu.
$$
More generally, for any finite-dimensional module $M$ on which $T$ acts semisimply, we denote by $M_\mu$ the $\mu$ weight space, and set $\op{wt}(M):= \{\mu \in \mc[P]: M_\mu \neq 0\}$ and $\op{ch}(M):= \sum_{\mu \in \op{wt}(M)} \left(\op{dim} M_\mu \right) e^\mu$ the set of weights and the character of $M$, respectively. 

Finally, we let $W$ be the Weyl group of $G$, generated by simple reflections $\{s_i\}_{i=1}^\ell$ associated to each simple root. Then $W$ naturally acts on the weight lattice $\mc[P]$, which we denote via $w(\lambda)$ for $w \in W$ and $\lambda \in \mc[P]$. A second action of $W$ on $\mc[P]$ is the \emph{shifted action} (also called the \emph{dot action}), which is given by 
$$
w \sbt \lambda := w(\lambda+\rho)-\rho, 
$$
where as before $\rho$ is the Weyl vector.

\subsection{Vector bundles on $G/B$ and the Borel--Weil--Bott theorem}
We now give a construction for vector bundles on $G/B$ via the association 
$$
\{\text{finite-dimensional } B\text{-modules}\} \iff \{\text{finite-dimensional } G\text{-equivariant bundles on } G/B\},
$$
which we sketch as follows. Let $M$ be a finite-dimensional $B$-module, and define the right action of $B$ on $G \times M$ by 
$$
(g, m).b := (gb, b^{-1}(m)).
$$
With this action, set $G \times_B M := (G \times M)/B$. Then for any equivalence class $[g, m] \in G \times_B M$, the projection map $G \times_B M \to G/B$ given by $[g, m] \mapsto gB$ is a vector bundle map with fiber $M$, which we denote as in the following definition.

\begin{definition} For a finite-dimensional $B$-module $M$, we denote by $\mc[L](M) := G \times_B M$ the vector bundle on $G/B$ as in the previous construction. 
\end{definition}

The total space $G \times_B M$ of $\mc[L](M)$ has a natural left $G$-action, given by $g.[g', m]:=[gg', m]$, which is compatible with the bundle map $G \times_B M \to G/B$. Thus, $\mc[L](M)$ is a $G$-equivariant vector bundle on $G/B$. Associating $\mc[L](M)$ with its sheaf of sections, we get a $G$-module structure on the cohomology groups $H^i (G/B, \mc[L](M))$. A central topic of interest is how these cohomology groups decompose as irreducible $G$-modules for various vector bundles $\mc[L](M)$. 

We next specialize to the case when $M$ is one-dimensional, so that $\mc[L](M)$ is a line bundle. Fix a weight $\lambda \in \mc[P]$, and extend this to a character of $B$ via $\lambda(N)=1$. Then set $\C_\lambda$ to be the $B$-module given by $b(z):=\lambda(b)z$. Using the general construction above, we construct $G$-equivariant line bundles $\mc[L](\lambda)$ on $G/B$ as follows. 

\begin{definition} The line bundle $\mc[L](\lambda)$ on $G/B$ is given by $\mc[L](\lambda):= \mathcal{L}(\C_{-\lambda}) = G \times_B \C_{-\lambda}$. 
\end{definition}

\begin{remark} We highlight that we have introduced a negative sign in the definition of $\mc[L](\lambda)$, for convenience in stating later results. 
\end{remark}

For these line bundles, the cohomology groups $H^i(G/B, \mc[L](\lambda))$ are completely understood by the work of Borel, Weil, and the later work of Bott. These results are collectively known as the Borel--Weil--Bott theorem, which we record here. 

\begin{theorem} \label{BWB} 
\begin{enumerate}
\item Let $\lambda \in \dom$ be a dominant integral weight, and let $w \in W$. Then 
$$
H^i(G/B, \mc[L](w \sbt \lambda)) = \begin{cases} V(\lambda)^\ast & i = l(w), \\ 0 & \text{otherwise,} \end{cases}
$$
where $l(w)$ is the length of $w$ in the Weyl group, and $V(\lambda)^\ast$ is the dual representation of $G$ to $V(\lambda)$. 

\item If $\lambda \in \mc[P]$ such that $w \sbt \lambda \not \in \dom$ for any $w \in W$, then $H^i(G/B, \mc[L](\lambda)) = 0$ for all $i \geq 0$. 
\end{enumerate}
\end{theorem}

In particular, for any $\lambda \in \mc[P]$ the line bundle $\mc[L](\lambda)$ has at most one nontrivial cohomology group, which has an explicit description as a $G$-representation. We restate this in the following nonstandard fashion, which will be useful when considering more general vector bundles.

\begin{corollary} \label{line-bundle-support}
 Let $H^\bullet(G/B, \mc[L](\lambda))$ be the total cohomology of $\mc[L](\lambda)$, and suppose $\mu \in \dom$ such that $V(\mu) \subset H^\bullet(G/B, \mc[L](\lambda))$ as a subrepresentation (hence $H^\bullet(G/B, \mc[L](\lambda))$ is nontrivial). Then $\mu^\ast \in W \sbt \lambda$, where $\mu^\ast := -w_0\mu$ is the highest weight of $V(\mu)^\ast$, $w_0 \in W$ is the longest word of the Weyl group, and $W \sbt$ denotes the shifted orbit under $W$. 
\end{corollary}

Now, we return to the case of vector bundles $\mc[L](M)$ on $G/B$. Unlike for line bundles, the total cohomology $H^\bullet(G/B, \mc[L](M))$ need not have only one nontrivial summand, nor is it in general the case that $H^i(G/B, \mc[L](M))$ is isomorphic to a single irreducible $G$-representation. Similarly, since a general $B$-module $M$ need not be completely reducible, we cannot typically decompose $\mc[L](M)$ as a direct sum of line bundles to immediately apply the Borel--Weil--Bott theorem. However, we can introduce a filtration $F^\bullet(M)$ on $M$ such that successive quotients \emph{are} completely-reducible $B$-modules; that is, for the associated graded module we have 
$$
\op{gr}(M) = \bigoplus_{\xi_i} \C_{\xi_i}
$$
for some multiset of weights $\{\xi_i\}$, where each $\xi_i \in \op{wt}(M)$. This allows us to prove the following useful lemma, generalizing Corollary \ref{line-bundle-support} to vector bundles, which gives candidates for irreducible representations $V(\mu)$ appearing in $H^\bullet(G/B, \mc[L](M))$. 

\begin{lemma} \label{cohom-support}
Let $M$ be a finite-dimensional $B$-module, and let $\mu \in \dom$ such that $V(\mu) \subset H^\bullet(G/B, \mc[L](M))$ is a subrepresentation. Then
$$
\mu^\ast \in W \sbt \{-\op{wt}(M)\}.
$$
\end{lemma}

\begin{proof}
As in the preceding discussion, fix a filtration $F^\bullet(M)$ of $M$ such that the associated graded $\op{gr}(M)$ is completely reducible, with $\op{gr}(M) = \bigoplus_{\xi_i} \C_{\xi_i}$. Then at the level of cohomology, this gives a surjection 
$$
\bigoplus_{\xi_i} H^\bullet(G/B, \mc[L](-\xi_i)) \onto H^\bullet(G/B, \mc[L](M));
$$
note again that by our conventions we introduce the signs in $\mc[L](-\xi_i)$. Then if $V(\mu)$ is a subrepresentation of $H^\bullet(G/B, \mc[L](M))$, it must also arise as a subrepresentation of $\bigoplus_{\xi_i} H^\bullet(G/B, \mc[L](-\xi_i))$. Since $V(\mu)$ is irreducible, in fact it is a subrepresentation of (and hence isomorphic to) a single $H^\bullet(G/B, \mc[L](-\xi_i))$, so that by Corollary \ref{line-bundle-support} we have $\mu^\ast \in W \sbt \{-\xi_i\} \subset W \sbt \{-\op{wt}(M)\}$, as desired. 
\end{proof}

\begin{remark} Typically, the map $\bigoplus_{\xi_i} H^\bullet(G/B, \mc[L](-\xi_i)) \onto H^\bullet(G/B, \mc[L](M))$ has nontrivial kernel, so that in general we do not have $H^\bullet(G/B, \mc[L](M)) \cong \bigoplus_{\xi_i} H^\bullet(G/B, \mc[L](-\xi_i))$. The kernel of this map can be computed using spectral sequence methods or the algorithm of Lachowska--Qi as implemented in \cite{HV}; however, we will not have need for that information here. 
\end{remark}

\subsection{Euler--Poincar\'e characteristic of $\mc[L](M)$} \label{eulercharsec}
For a finite-dimensional $G$-module $V$, we denoted by 
$$
\op{ch}(V) = \sum_{\mu \in \op{wt}(V)} \left( \op{dim}(V_\mu)\right) e^\mu
$$ 
the $T$-character of $V$. We consider characters $\op{ch}(V)$ to live in the ring $\Z[\mc[P]]$, where we associate $\mu \in \mc[P]$ to $e^\mu$. Later, we will need also the involution $\diamond: \Z[\mc[P]] \to \Z[\mc[P]]$, which is given by $\diamond(e^\mu) := e^{-\mu}$ and extended linearly.

 As we would like to understand the representation structure of the cohomology groups $H^i(G/B, \mc[L](M))$, a natural approach would be to consider their characters. However, like the representation themselves, we do not have a clear way to compute $\op{ch}(H^i(G/B, \mc[L](M)))$ directly. Nevertheless, there is a uniform computation for their Euler--Poincar\'e characteristic, which we define below. 

\begin{definition} Given a vector bundle $\mc[L](M)$ on $G/B$, the Euler--Poincar\'e characteristic of $\mc[L](M)$ is given by 
$$
\chi_T(G/B, \mc[L](M)):= \sum_{i \geq 0} (-1)^i \op{ch}(H^i(G/B, \mc[L](M))) \in \Z[\mc[P]].
$$
We abuse notation and also write $\chi_T (\mc[L](M))$, when no confusion arises. 
\end{definition}

While seemingly more involved, the Euler--Poincar\'e characteristic $\chi_T(\mc[L](M))$ is much easier to compute than the character of $H^\bullet(G/B, \mc[L](M))$. This comes through the work of Demazure \cite{Dem} on the more general setting of line and vector bundles on Schubert varieties, of which $G/B$ is a particular example. We do not develop this narrative in its entirety (see for instance \cite{Ku2}*{Chapter 8} and references therein), but only introduce what is necessary for our purposes. We begin by recalling the simple Demazure operators and their action on $\Z[\mc[P]]$. 

\begin{definition} Let $1 \leq i \leq \ell$. Then the Demazure operator $D_i: \Z[\mc[P]] \to \Z[\mc[P]]$ acts via 
$$
D_i(f):= \frac{f-e^{-\alpha_i}s_i(f)}{1-e^{-\alpha_i}}. 
$$
\end{definition}

One can check that the image of $D_i$ is again in $\Z[\mc[P]]$ and that $D_i^2=D_i$. A less immediate but classical property is that the collection of operators $\{D_i\}_{i=1}^\ell$ satisfy the braid relations coming from the Weyl group $W$ (cf. \cite{Ku2}*{Corollary 8.2.10}). Thus, for any $w \in W$, we get a well-defined operator $D_w$ by considering reduced expressions: if $w=s_{i_1} s_{i_1} \cdots s_{i_k}$ is a reduced word, then we can set $D_w:=D_{i_1}D_{i_2}\cdots D_{i_k}$. 

For $w_0 \in W$ the longest word, fix the operator $D_{w_0}$. Recall the involution $\diamond$ on $\Z[\mc[P]]$ from the beginning of this subsection. Then the following proposition is a special case of the more general \emph{Demazure character formula}. 

\begin{proposition} \label{Euler-char}
Let $M$ be a $B$-module and $\mc[L](M)$ the associated vector bundle on $G/B$. Then the Euler--Poincar\'e characteristic $\chi_T(\mc[L](M))$ is given by 
$$
\chi_T(\mc[L](M)) = \diamond(D_{w_0}(\diamond(\op{ch}(M)))).
$$
\end{proposition} 

Of particular importance, when applied to the line bundle $\mc[L](\lambda)$ for dominant integral $\lambda \in \dom$, Proposition \ref{Euler-char} together with the Borel--Weil--Bott theorem as in Theorem \ref{BWB} gives that 
$$
\op{ch}(H^0(G/B, \mc[L](\lambda))) = \diamond(D_{w_0}(e^\lambda)) = \op{ch}(V(\lambda)^\ast).
$$

A utility of Demazure operators are their applicability and generalizations to many character computation problems. We will see in Section \ref{embedding}, for example, another application of Demazure operators for computing characters of tensor products of representations that will be of particular use.

\section{Hochschild cohomology of $G/B$ and the sheaf of polyvector fields} \label{Hoch}

In this section, we give a brief introduction to the primary geometric object of interest for this note: the Hochschild cohomology $HH^\bullet(X)$ of a smooth projective variety $X$. We make no effort for a complete account of this topic, but instead give a self-contained overview of the key properties and results which we will apply as a black box in the setting of flag varieties. We refer interested readers to \cite{BS} and references therein, on which this section is closely modeled. 

Hochschild cohomology was initially introduced as a tool for understanding the deformation theory of associative algebras; this machinery has since been adapted to various suitable categorical settings. For the purposes of this note, we work with the following definition of Hochschild cohomology in algebraic geometry.  

\begin{definition} Let $X$ be a smooth, projective variety, and denote by $\Delta: X \hookrightarrow X \times X$ the diagonal embedding. Then the $i^{th}$ Hochschild cohomology of $X$ is defined via 
$$
HH^i(X):= \op{Ext}^i_{X \times X}( \Delta_\ast \mathcal{O}_X, \Delta_\ast \mathcal{O}_X),
$$
where $\mathcal{O}_X$ is the structure sheaf on $X$. We also denote by 
$$
HH^\bullet(X) := \bigoplus_{i \geq 0} HH^i(X)
$$
the total Hochschild cohomology. 
\end{definition}

An additional key feature of Hoschschild cohomology is of an algebraic nature; in particular, it has a structure of a Gerstenhaber algebra. Gerstenhaber algebras have a graded-commutative product as well as a Lie bracket of degree $-1$, with some compatibilities. We focus here only on the Lie bracket, which gives a mapping
$$
HH^i(X) \times HH^j(X) \to HH^{i+j-1}(X);
$$
further, restricting this product to $HH^1(X) \times HH^1(X) \to HH^1(X)$ gives a Lie algebra structure on $HH^1(X)$. Then as the product in general includes the shift in grading, each $HH^i(X)$ is a Lie algebra module over $HH^1(X)$. 

While this definition of Hochchild cohomology is quite general, we will work instead with another description of $HH^\bullet(X)$ that will be more approachable when considering the flag variety. This comes via the following proposition, called the Hochschild--Kostant--Rosenberg (or HKR) isomorphism, who proved this result in the affine case \cite{HKR}. For projective varieties and related generalizations, this was adapted by many authors (see \cite{BS}*{Section 2.4} for select references). 

\begin{proposition} \label{HKRvec}
Let $X$ be a smooth projective variety with tangent bundle $\mathcal{T}_X$. Then there is a vector space isomorphism
$$
HH^i(X) \cong \bigoplus_{p+q=i} H^q(X, \bigwedge\nolimits^{\! p} \mathcal{T}_X).
$$
In particular, $HH^\bullet(X) \cong H^\bullet(X, \bigwedge^\bullet \mathcal{T}_X)$.
\end{proposition}

We refer to $\bigwedge^\bullet \mathcal{T}_X$ as the sheaf of polyvector fields on $X$. Recall that $H^\bullet(X, \bigwedge^\bullet \mathcal{T}_X)$ also has a natural Gerstenhaber algebra structure coming from the Schouten--Nijenhuis bracket. However, the HKR isomorphism is \emph{not} in general an algebra isomorphism. Nevertheless, by appropriately ``twisting" the standard HKR isomorphism (due to an insight of Kontsevich) one arrives at an isomorphism of Gerstenhaber algebras; we make use of the existence of such a Gerstenhaber algebra isomorphism, and refer to \cite{CRVdB1}, \cite{CRVdB2} for a more comprehensive treatment.

We specialize now to the case $X=G/B$ for the remainder of this paper, and we focus our attention on the realization of $HH^\bullet(G/B)$ as the cohomology of the sheaf of polyvector fields. We have, using the general machinery of Section \ref{bundles}, the identifications 
$$
\mathcal{T}_{G/B} \cong G \times_B (\mf[g]/\mf[b]) \cong G \times_B \mf[u] =: \mathcal{L}(\mf[u]),
$$
so that $\bigwedge^p \mathcal{T}_{G/B} \cong \mathcal{L}(\bigwedge^p \mf[u])$ for all $p \geq 0$, and in total $\bigwedge^\bullet \mathcal{T}_{G/B} \cong \mathcal{L}(\bigwedge^\bullet \mf[u])$. Further, using the HKR isomorphism we see that 
$$
HH^1(G/B) \cong H^0(G/B, \mathcal{T}_{G/B}).
$$
We record now the following classical result of Bott \cite{Bo} which allows us to identify this cohomology. While Bott's result is on the module structure, the following isomorphism also holds as Lie algebras by checking the appropriate bracket coming from the Gerstenhaber algebra structure.

\begin{proposition} \label{BottProp} For all $i \geq 1$, 
$
H^i(G/B, \mathcal{T}_{G/B}) = 0,
$
and $H^0(G/B, \mathcal{T}_{G/B}) \cong \mf[g].
$
\end{proposition}

By the general vector bundle framework, $H^\bullet(G/B, \bigwedge^\bullet \mathcal{T}_{G/B})$ is a $G$-representation, and equivalently a $\mf[g]$-representation since $G$ is simply connected. Proposition \ref{BottProp} leads to another construction of the $\mf[g]$-action on $H^\bullet(G/B, \bigwedge^\bullet \mathcal{T}_{G/B})$, via the Gerstenhaber bracket; it is important to note that these two actions are in fact the same. Transporting the latter perspective via the enhanced HKR (algebra) isomorphism, we see that $HH^1(G/B) \cong \mf[g]$ acts on $HH^\bullet(G/B)$. 

In total, to understand $HH^\bullet(G/B)$ as a $\mf[g]$-representation, we will utilize the more familiar setting of vector bundles on flag varieties by considering the structure of the representation $H^\bullet(G/B, \bigwedge^\bullet \mathcal{T}_{G/B})$. We take up this approach in the following section.

\section{An embedding of $V(\rho) \otimes V(\rho)$ and components of $HH^\bullet(G/B)$} \label{embedding}

For two dominant integral weights $\lambda, \mu \in \dom$, the tensor decomposition problem seeks to understand the irreducible summands in $V(\lambda) \otimes V(\mu)$:
$$
V(\lambda) \otimes V(\mu) \cong \bigoplus_{\nu \in \dom} V(\nu)^{\oplus m_{\lambda, \mu}^\nu},
$$
where we denote by $m_{\lambda, \mu}^\nu \in \Z_{\geq 0}$ the multiplicity with which $V(\nu)$ appears in the tensor product. Equivalently, we can consider at the level of characters 
$$
\op{ch}(V(\lambda) \otimes V(\mu)) = \op{ch}(V(\lambda)) \op{ch}(V(\mu)) = \sum_{\nu \in \dom} m_{\lambda, \mu}^\nu \op{ch}(V(\nu)).
$$
While reducing the tensor decomposition problem to understanding products of characters is equivalent, it is in general no less complicated. Surprisingly, the language of characters gives a succinct method to connect the tensor decomposition problem for $V(\rho) \otimes V(\rho)$ as in Conjecture \ref{Kconj} to the cohomology $H^\bullet(G/B, \bigwedge^\bullet \mathcal{T}_{G/B})$. To do this, we need to better understand the characters of $V(\rho)$ and its tensor square. 

\subsection{A compact character formula for $V(\rho) \otimes V(\rho)$} To this end, we recall the following result of Brauer, recorded in \cite{Kos2}*{Section 4.8} (see also \cite{Ku1}*{Theorem 2.14}).

\begin{proposition} \label{tensorcharacter}
For $\lambda, \mu \in \dom$, we have 
$$
\op{ch}(V(\lambda) \otimes V(\mu)) = D_{w_0}(e^\lambda D_{w_0}(e^\mu)) = D_{w_0}(e^\lambda \op{ch}(V(\mu))).
$$
\end{proposition} 

Aiming for the present case of interest $V(\rho) \otimes V(\rho)$, we make use of the following lemma on $\op{ch}(V(\rho))$, which is an easy exercise following from the Weyl character formula (cf. \cite{Kac}*{Exercise 10.1}).

\begin{lemma} \label{rhochar}
$\op{ch}(V(\rho)) = e^{\rho} \prod_{\beta \in \Phi^+} (1+e^{-\beta}).$
\end{lemma}

Together, these give the following corollary, a compact formula for $\op{ch}(V(\rho) \otimes V(\rho))$. 

\begin{corollary} \label{rhosquarechar}
$\op{ch}(V(\rho) \otimes V(\rho)) = D_{w_0} \left( \prod_{\beta \in \Phi^+} (1+e^\beta) \right).$
\end{corollary}

\begin{proof}
Specializing Proposition \ref{tensorcharacter} to $\lambda=\mu=\rho$, we get 
$$
\op{ch}(V(\rho) \otimes V(\rho)) = D_{w_0}(e^\rho \op{ch}(V(\rho))).
$$
But by Lemma \ref{rhochar} we have that $e^\rho \op{ch}(V(\rho)) = e^{2\rho} \prod_{\beta \in \Phi^+} (1+e^{-\beta}) = \prod_{\beta \in \Phi^+} (1+e^\beta)$, using that $2\rho = \sum_{\beta \in \Phi^+} \beta$.
\end{proof}

Finally, note that (since in our conventions $\mf[n]$ corresponds to the positive roots) the $B$-module $\bigwedge^\bullet \mf[n]$ satisfies
$$
\op{ch}\left(\bigwedge \nolimits^{\! \bullet} \mf[n]\right) = \prod_{\beta \in \Phi^+} (1 + e^\beta).
$$
Combined with the previous result, we get the following corollary; this is the key observation linking $V(\rho) \otimes V(\rho)$ and $HH^\bullet(G/B)$. 

\begin{corollary} \label{keyremark}
 $\op{ch}(V(\rho) \otimes V(\rho)) = D_{w_0}(\op{ch}(\bigwedge^\bullet \mf[n]))$.
\end{corollary}

\subsection{Components of $HH^\bullet(G/B)$} 
We are now prepared to prove the main results of this note. As a preliminary, we give the following lemma which constrains the highest weights in irreducible subrepresentations of $HH^\bullet(G/B)$. 

\begin{lemma} \label{supset}
Let $\lambda \in \dom$ be a dominant integral weight such that $V(\lambda) \subset H^\bullet(G/B, \bigwedge^\bullet \mathcal{T}_{G/B})$. Then $\lambda \leq 2\rho$ in the dominance order. 
\end{lemma}

\begin{proof}
As $H^\bullet (G/B, \bigwedge^\bullet \mathcal{T}_{G/B}) = H^\bullet(G/B, \mathcal{L}(\bigwedge^\bullet \mf[u]))$, by Lemma \ref{cohom-support} we know that $V(\lambda) \subset H^\bullet(G/B, \mathcal{L}(\bigwedge^\bullet \mf[u]))$ forces 
$$
\lambda^\ast = -w_0\lambda \in W \sbt \{- \op{wt}(\bigwedge \nolimits^{\! \bullet} \mf[u])\}.
$$
Now, as $\mf[u]$ is associated to the negative roots, we have that $-\op{wt}(\bigwedge^\bullet \mf[u]) = \op{wt}(\bigwedge^\bullet \mf[n])$. By the proof of Corollary \ref{rhosquarechar} and the following line, these are precisely the weights appearing in $e^\rho \op{ch}(V(\rho))$. Thus, write $\lambda^\ast$ as 
$$
\lambda^\ast = w \sbt(\rho +\gamma)
$$
for some $w \in W$ and $\gamma \in \op{wt}(V(\rho))$. Then 
$$
\begin{aligned}
2\rho - \lambda^\ast &= 2\rho - \left( w \sbt( \rho +\gamma) \right) \\
&= 2\rho - \left( w(2\rho+\gamma)-\rho \right) \\
&= (2\rho - w(2\rho)) + (\rho - w(\gamma));
\end{aligned}
$$
this is in $\mathcal{Q}^+$, as $w(2\rho) \leq 2\rho$ for any $w$, and $w(\gamma) \leq \rho$ since $\gamma \in \op{wt}(V(\rho))$ (hence $w(\gamma) \in \op{wt}(V(\rho))$). Therefore $\lambda^\ast \leq 2\rho$, and since $\rho$ is self-dual (i.e. $\rho=\rho^\ast$) this also gives $\lambda \leq 2\rho$, as desired. 
\end{proof}

We can now give the proof of Proposition \ref{embeddingprop} of the introduction. 

\begin{proposition} \label{embedprop2}
For $G$ a simple, simply-connected complex algebraic group, we have $V(\rho) \otimes V(\rho) \hookrightarrow H^\bullet(G/B, \bigwedge^\bullet \mathcal{T}_{G/B})$. 
\end{proposition}

\begin{proof}
First, we consider the character $\op{ch}(V(\rho) \otimes V(\rho))$. Since $\rho$ is self-dual, we have that 
$$
\op{ch}(V(\rho) \otimes V(\rho)) = \diamond \op{ch}(V(\rho) \otimes V(\rho)),
$$
where $\diamond$ is the involution on $\Z[\mathcal{P}]$ as in Section \ref{eulercharsec}. Therefore, by Corollary \ref{rhosquarechar} and Proposition \ref{Euler-char} we have that 
$$
\begin{aligned}
\op{ch}(V(\rho) \otimes V(\rho)) &= \diamond \op{ch}(V(\rho) \otimes V(\rho)) \\
&= \diamond D_{w_0} \left( \prod_{\beta \in \Phi^+} (1+ e^\beta) \right) \\
&= \diamond D_{w_0} \left( \diamond \prod_{\beta \in \Phi^+} (1+e^{-\beta}) \right) \\
&= \diamond D_{w_0} \left( \diamond \op{ch}(\bigwedge \nolimits^{\! \bullet} \mf[u]) \right) \\
&= \chi_T \left( \bigwedge \nolimits^{\! \bullet} \mathcal{T}_{G/B} \right)
\end{aligned}
$$
Reinterpreted, this says that as virtual modules (in the representation ring of $G$), we get that 
$$
V(\rho) \otimes V(\rho) = \sum_{i \geq 0} (-1)^i H^i(G/B, \bigwedge \nolimits^{\! \bullet} \mathcal{T}_{G/B});
$$
since $V(\rho) \otimes V(\rho)$ is an honest (not just virtual) module, all ``negative" irreducible components in the odd terms $(-1)^{2j+1} H^{2j+1}(G/B, \bigwedge^\bullet \mathcal{T}_{G/B})$ must cancel completely with components coming from the even terms $H^{2k}(G/B, \bigwedge^\bullet \mathcal{T}_{G/B})$. That is, we can find some subcollection $\{V(\nu_i)\}$ of irreducible components in the decomposition of $H^\bullet(G/B, \bigwedge^\bullet \mathcal{T}_{G/B})$ such that 
$$
V(\rho) \otimes V(\rho) \cong \bigoplus_{\nu_i} V(\nu_i) \subset H^\bullet(G/B, \bigwedge \nolimits^{\! \bullet} \mathcal{T}_{G/B}),
$$
as desired. 
\end{proof}

\begin{remark} 
We emphasize that in the above argument, the Euler--Poincar\'e characteristic is given by alternating sums in $H^i(G/B, \bigwedge^\bullet \mathcal{T}_{G/B})$ and \emph{not} in $HH^i(G/B)$, which is generally distinct as seen by the HKR isomorphism in Proposition \ref{HKRvec}. It would be interesting to find a representation-theoretic or algebraic formulation for the Euler--Poincar\'e characteristic with respect to the Hochschild grading.
\end{remark}

\begin{remark}
Note that the proof of Proposition \ref{embedprop2} does not rely on Lemma \ref{supset}; however, without the latter we could a priori have components in $H^\bullet(G/B, \bigwedge^\bullet \mathcal{T}_{G/B})$ of highest weight $\lambda \not \leq 2\rho$ which cancel in the Euler--Poincar\'e characteristic, so would not be reflected in $V(\rho) \otimes V(\rho)$. 
\end{remark}

\begin{example} 
\begin{enumerate}

\item[(a)] Let $G=SL_3(\C)$. Then using the algorithm implemented by Hemelsoet--Voorhaar \cite{HV} in Sage \cite{TSD} (or by the discussion below in Section \ref{furtherqs}), we find the following nontrivial cohomologies:
$$
\begin{aligned}
H^0(G/B, \bigwedge \nolimits^{\! 0} \mathcal{T}_{G/B}) &\cong V(0) = \C, \\
H^0(G/B, \bigwedge \nolimits^{\! 1} \mathcal{T}_{G/B}) &\cong V(\rho), \\
H^0(G/B, \bigwedge \nolimits^{\! 2} \mathcal{T}_{G/B}) &\cong V(\rho) \oplus V(3\varpi_1) \oplus V(3\varpi_2), \\
H^0(G/B, \bigwedge \nolimits^{\! 3} \mathcal{T}_{G/B}) &\cong V(2\rho).
\end{aligned}
$$
In this example, all higher cohomologies vanish, so that the Hochschild cohomology satisfies $HH^i(G/B) = H^0(G/B, \bigwedge^i \mathcal{T}_{G/B})$ for all $i$; this property is called \emph{Hochschild global} \cite{BS} or \emph{Hochschild affine} \cite{HV}. While rare among complete flag varieties, the proof of Proposition \ref{embedprop2} implies 
$$
V(\rho) \otimes V(\rho) \cong HH^\bullet(G/B)
$$
for all Hochschild global $G/B$, such as $SL_3(\C)/B$.

\item[(b)] Let $G$ be of type $D_4$. Then again by the algorithm of \cite{HV} we find that $H^\bullet(G/B, \bigwedge^\bullet \mathcal{T}_{G/B})$ has $652$ irreducible components. The only components appearing in higher cohomology are
$$
\begin{aligned}
H^1(G/B, \bigwedge \nolimits^{\! 3} \mathcal{T}_{G/B}) &= V(\varpi_2), \text{ and} \\
H^1(G/B, \bigwedge \nolimits^{\! 4} \mathcal{T}_{G/B}) &= V(2\varpi_2).
\end{aligned}
$$

One can also find (again using programs such as Sage or LiE) that the tensor product $V(\rho) \otimes V(\rho)$ has 648 irreducible components. Thus in the Euler--Poincar\'e calculation as in Proposition \ref{embedprop2}, the two irreducible components coming from higher cohomology are cancelled by copies appearing in $H^0(G/B, \bigwedge^\bullet \mathcal{T}_{G/B})$, accounting for the difference of four irreducible components between $H^\bullet(G/B, \bigwedge^\bullet \mathcal{T}_{G/B})$ and $V(\rho) \otimes V(\rho)$. All together,
$$
H^\bullet(G/B, \bigwedge \nolimits^{\! \bullet} \mathcal{T}_{G/B}) \cong \left(V(\rho) \otimes V(\rho) \right) \oplus V(\varpi_2)^{\oplus 2} \oplus V(2\varpi_2)^{\oplus 2}.
$$
\end{enumerate}
\end{example}

Finally, combining Lemma \ref{supset} and Proposition \ref{embedprop2} we deduce the following theorem. Recall that we denote by $m_\lambda$ the multiplicity of the representation $V(\lambda)$ in the tensor product $V(\rho) \otimes V(\rho)$. 

\begin{theorem} \label{maintheorem2}
Let $G$ be a simple, simply-connected complex algebraic group. Assuming the validity of Kostant's Conjecture \ref{Kconj}, for a dominant integral weight $\lambda \in \dom$ we have that 
$$
V(\lambda) \subset H^\bullet(G/B, \bigwedge \nolimits^{\! \bullet} \mathcal{T}_{G/B})
$$ 
if and only if $\lambda \leq 2\rho$ in the dominance order, and appears with nonzero multiplicity at least $m_\lambda$.  
\end{theorem}

\begin{proof}
By Lemma \ref{supset}, if $V(\lambda) \subset H^\bullet(G/B, \bigwedge^\bullet \mathcal{T}_{G/B})$ then necessarily $\lambda \leq 2\rho$. Then by Conjecture \ref{Kconj}, $V(\lambda)$ also appears in $V(\rho) \otimes V(\rho)$ with nonzero multiplicity $m_\lambda$, so that by the embedding of Proposition \ref{embedprop2} we get that $V(\lambda)$ appears with at least this multiplicity.

Conversely, if $\lambda \leq 2\rho$, then Conjecture \ref{Kconj} says that $V(\lambda)$ appears in $V(\rho) \otimes V(\rho)$ with nonzero multiplicity $m_\lambda$; again by Proposition \ref{embedprop2} we get that $V(\lambda)$ must appear in $H^\bullet(G/B, \bigwedge^\bullet \mathcal{T}_{G/B})$ with at least this multiplicity. 
\end{proof}

	\begin{remark} \label{SimpConn} 
	For $G$ arbitrary simple (not necessarily simply-connected), $V(\rho)$ might not be a representation of $G$. However, since all simple factors of $ H^\bullet(G/B, \bigwedge \nolimits^{\! \bullet} \mathcal{T}_{G/B})$ satisfy $\lambda \leq 2\rho$ by Lemma \ref{supset} (so are in the root lattice), Theorem \ref{maintheorem2} remains valid for all simple $G$.
		\end{remark}

\subsection{Further questions} \label{furtherqs}
We conclude with some final remarks and further questions of interest. For simplicity, we assume throughout this subsection that Conjecture \ref{Kconj} holds, for ease of statements. 

Recall that each $HH^i(G/B)$ is itself a $\mf[g]$-representation. While Theorem \ref{maintheorem2} gives a complete list of irreducible components in $HH^\bullet(G/B)$, it gives no concrete information on how these are distributed throughout the various graded pieces. In fact, we can further ask about the finer decomposition coming from the sheaf of polyvector fields:

 \begin{question} Which irreducible components appear in the various $HH^i(G/B)$? In $H^q(G/B, \bigwedge^p \mathcal{T}_{G/B})$?
  \end{question} 

\noindent Of course, understanding the second question suffices for understanding the first, by the HKR isomorphism. Using the standard Borel--Weil--Bott approach as in Lemma \ref{cohom-support}, we can constrain the highest weights $\lambda$ appearing in the latter cohomology groups by looking at the shifted action of elements $w \in W$ with a fixed length; however, this would give only potential candidates for irreducible components in $H^q(G/B, \bigwedge^p \mathcal{T}_{G/B})$ and does not guarantee these must appear (although Theorem \ref{maintheorem2} asserts they must appear \emph{somewhere} in the total cohomology). 

At the ``extreme" ends, these cohomologies are straightforward to compute: of course, $H^0(G/B, \bigwedge^0 \mathcal{T}_{G/B})$ is the trivial module and all higher cohomologies vanish, and $H^q(G/B, \mathcal{T}_{G/B})$ is fully described by Proposition \ref{BottProp}. At the other extreme, setting $n:= \op{dim}(G/B)$ we know that $\bigwedge^n \mathcal{T}_{G/B} \cong \mathcal{L}(2\rho)$ is a line bundle, so we can simply apply the Borel--Weil--Bott theorem. Finally, for the same $n$ using the isomorphism 
$$
\bigwedge \nolimits^{\! n-1} \mathcal{T}_{G/B} \cong \Omega_{G/B}^1 \otimes \omega_{G/B}^\vee \cong \Omega_{G/B}^1 \otimes \mathcal{L}(2\rho),
$$
a result of Wahl \cite{Wah}*{Proposition 3.9} shows that $H^q(G/B, \bigwedge^{n-1} \mathcal{T}_{G/B}) =0$ for $q \geq 1$ and 
$$
H^0(G/B, \bigwedge \nolimits^{\! n-1} \mathcal{T}_{G/B}) \cong \bigoplus_{\substack{\beta \in \Phi^+ \\ 2\rho-\beta \in \dom}} V(2\rho -\beta).
$$
The cohomologies for $2 \leq p \leq \op{dim}(G/B)-2$ and arbitrary $q$ are more subtle.

A second natural question is how much of the current work is adaptable to the setting of partial flag varieties $G/P$, where $B \subset P \subset G$ is a parabolic subgroup. For certain partial flag varieties--referred to as \emph{generalized Grassmannians}--this was considered by Belmans--Smirnov \cite{BS}. There, they give a precise representation-theoretic description of each $HH^i(G/P)$ when $G/P$ is either cominuscule or adjoint; in this geometrically advantageous setting, the cohomology is concentrated in $H^0(G/B, \bigwedge^i \mathcal{T}_{G/P})$. Further, the highest weight representations appearing in each $HH^i(G/P)$ in loc. cit. have a nice combinatorial description based on Kostant's work on Lie algebra cohomology \cite{Kos3}.

For arbitrary $G/P$, note that the weights of $\bigwedge^\bullet \mathcal{T}_{G/P}$ are sums of unique negative roots in $\Phi^- \backslash \Phi^-_{L}$, where $\Phi_L$ is the set of roots for the Levi component of $P$. The lowest weight in dominance order in this set is 
$$
\sum_{\beta \in \Phi^- \backslash \Phi^-_{L}} \beta =: -(2\rho - 2\rho_L),
$$
with $\rho_L$ being the half-sum of positive roots in $\Phi^+_L$; when $P=B$ and $\Phi_L=\varnothing$, this corresponds to $-2\rho$. Note that $2\rho-2\rho_L$ is a dominant integral weight. 

Unfortunately, it is \emph{not} the case that, for $\lambda \in \dom$ a dominant integral weight, $V(\lambda) \subset HH^\bullet(G/P)$ if and only if $\lambda \leq 2\rho-2\rho_L$ when $P \neq B$. Counterexamples exist already for Grassmannians; see for example \cite{BS}*{Table 5}, where $2\rho-2\rho_L = 4 \varpi_2$, and $2\varpi_2 \leq 4\varpi_2$ but $V(2\varpi_2) \not \subset HH^\bullet(G/P)$. Nevertheless, it is clear from the description in \cite{BS}*{Theorem B} that all $\lambda$ for which $V(\lambda)$ \emph{does} appear in $HH^\bullet(G/P)$ for cominuscule $G/P$ satisfy $\lambda \leq 2\rho-2\rho_L$; the failure then is the sufficiency of this condition. In the tensor decomposition problem, we expect dominance order to be far from sufficient for predicting irreducible components. We do not know of a suitable approach to understanding $HH^\bullet(G/P)$ generally via tensor product embeddings as in Proposition \ref{embedprop2}, and suspect that other methods would be more beneficial. With this goal in mind, we end with the following natural question which seeks a uniform perspective for Theorem \ref{maintheorem2} and the results of \cite{BS}.

\begin{question} Can the set $\{\lambda \in \dom: V(\lambda) \subset HH^\bullet(G/P)\}$ for general $G$ and $P$ be described combinatorially and uniformly?
\end{question}

\end{document}